\documentclass[12pt]{article}
\usepackage{amsmath,amsthm,amssymb}
\usepackage{graphicx}
\usepackage{wrapfig}

\begin{document}

\newtheorem{theorem}{Theorem}[section]
\newtheorem{lemma}{Lemma}[section]
\newtheorem{corollary}{Corollary}[section]
\newtheorem{remark}{Remark}[section]
\newtheorem{definition}{Definition}[section]

\newcommand{\1}{\hat{1}}
\newcommand{\0}{\hat{0}}
\newcommand{\Co}{\mathcal{C}}
\newcommand{\Sub}{\mathcal{L}}
\newcommand{\Sco}{\mathcal{S}}
\newcommand{\im}{\mathop{\rm Im}\nolimits}
\newcommand{\Hdim}{\mathop{\rm Hdim}\nolimits}
\newcommand{\Sz}{\mathop{\rm Sz}\nolimits}
\renewcommand{\leq}{\leqslant}
\renewcommand{\geq}{\geqslant}

\title{Homotopy types of group lattices}
\author{Kramarev I.P., Lokutsievskiy L.V.\date{}}

\maketitle

\begin{abstract}
  In this article we study group lattices using the ideas by K.S.Brown and D.Quillen of associating a certain topological space to a partially ordered set. We determine the exact homotopy type for the subgroup lattice of $PSL(2,7)$, find a connection between different group lattices and obtain some estimates for the Betty numbers of these lattices using the spectral sequence method.
\end{abstract}

\section{Introduction}

Let $P$ be a finite partially ordered set (poset, for short). One can associate a simplicial complex with $P$ in a canonical way.

\begin{definition}
  Let $\Delta P$ denote a simplicial complex with a vertex set $P$ consisting of such simplices $h_1h_2\ldots h_k$ that for some permutation $\sigma\in S_{k+1}$ we have $h_{\sigma(1)}<h_{\sigma(2)}<\ldots<h_{\sigma(k)}$ in $P$.
\end{definition}

Let $P$ and $Q$ be posets. A map $f$: $P\to Q$ is called a morphism (or map) of posets if it preserves the non-strict order, i.e. for all $x$, $y \in P$ such that $x\leq y$ we have $f(x) \leq f(y)$ in $Q$. The map $f$ induces a simplicial map $\Delta f$ from $\Delta P$ into $\Delta Q$ by an obvious rule. Thus any poset inclusion $P\subseteq Q$ induces an inclusion map of the associated simplicial complexes: $\Delta P\subseteq \Delta Q$.

Following D.Quillen we use the construction $P\to\Delta P$ to assign topological concepts to posets. For example, we call $P$ contractible provided $\Delta P$ is contractible, and we define the homology groups of $P$ to be those of $\Delta P$.

Consider a finite group $G$. The set of all subgroups of $G$ ordered by inclusion forms a lattice with a proper part $\Sub G=\{H \mid 1<H<G\}$. The set of cosets of all subgroups (including $\emptyset$ and $G$) ordered by inclusion also forms a lattice and $\Co G = \{xH \mid H<G,\ x\in G\}$ is its proper part (see \cite{Brown}). Thus the natural question arises: what homotopy type can spaces $\Delta\Sub G$ and $\Delta\Co G$ have for arbitrary finite group $G$?
 
K.Brown, C.Kratzer and J.Thevenaz proved that if $G$ is solvable, then both $\Sub G$ and$\Co G$ are homotopy equivalent to wedges of equidimensional spheres (see \cite{KraThev,Brown}). This fact is really intriguing, because a poset not associated to a finite group can have almost any homotopy type. Namely, for any finite simplicial complex $X$ there exists a finite poset $P$ such that $X$ and $\Delta P$ are homotopy equivalent (see \cite{Quillen}).

The problem of determining the homotopy type of $\Sub G$ and $\Co G$ for any finite group $G$ is still open. The case of simple groups seem to offer the main difficulty. The homotopy complementation formula by Bj\"orner and Walker (see \cite{BjornerWelker}) and similar results which allowed to compute the exact homotopy type of lattices of solvable groups depend on the existence of a normal subgroup and therefore cannot be used in this case.

Studying the shellability property of subgroup lattices of finite groups Shareshian proved that for certain series of finite simple groups ($L_2(p)$ for prime $p\equiv 3,5\pmod{8}$, $L_2(2^p)$, $L_2(3^p)$ and $Sz(2^p)$ for prime $p$) the homotopy type of lattice $\Sub G$ is that of a wedge of $|G|$ circles (see \cite{Shareshian}). The proof was based on the fact that in each case a subgroup lattice can be reduced to a 1-dimensional connected using Quillen's fiber lemma. Any such complex is obviously homotopy equivalent to a wedge of circles.

Therefore, another question may be of interest: does there exist a minimal simple group whose subgroup lattice has homotopy type different from a wedge of equidimensional spheres? We give an example of such group by determining the exact type of $\Sub PSL(2,7)$ which is a wedge of 48 circles and 48 spheres.

In this work we also use spectral sequence method to obtain estimates for Betti numbers of complexes $\Sub G$ and $\Co G$ for any finite group $G$ as well as more precise results for certain groups ($\Co PSL(2,7)$, $\Sub\Sz(2^{pq})$ and $\Sub\Sz(2^{p^k})$, for prime $p$ and $q$).

\section{Homotopy Methods}

For an element $h\in P$ we will use the following notation: $P_{<h}=\{x\in P: x<h\}$ (posets $P_{\leq h}$, $P_{>h}$ and $P_{\geq h}$ are defined similarly), $P_{\neq h}=\{x\in P: x\neq h\}$ and $P_{\notin M}=P\setminus M$ for any subset $M\subseteq P$.

The main tools for dealing with topological properties of posets are the Quillen's Fiber Lemma and the Homotopy Complementation Formula by Bj\"orner and Walker:

\begin{lemma}[Quillens' Fiber Lemma, \cite{Quillen}]
  Let $f$: $P \to Q$ be a map of finite posets such that upper fibers $f^{-1}(Q_{\geq x})$ are contractible for all $x\in Q$ (respectively, lower fibers $f^{-1}(Q_{\leq x})$ are contractible for all $x\in Q$). Then $f$ induces a homotopy equivalence between $P$ and $Q$.
\end{lemma}

\begin{definition}
  The join (or the least upper bound) of elements $x$ and $y$ of a poset $P$ is defined as an element $x \lor y = \inf\limits_z \{z \mid x,y \leq z\}$ (it it exists). The meet (or the greatest lower bound) $x \land y$ is defined similarly.
  
  A poset $P$ is called a lattice if for any elements $x$ and $y$ of $P$ there exist $x \lor y$ and $x \land y$ in $P$.

  A lattice is bounded provided it contains the greatest element $\1$ and the least element $\0$. Obviously, every finite lattice is bounded. The proper part of a bounded lattice $L$ is a subposet $\overline{L}=L\setminus\{\0,\1\}$.
\end{definition}

It is easy to check that if a poset $P$ contains an element $h_0$ which is comparable to all element of $P$ (e.g. $\hat 0$ or $\hat 1$), then $\Delta P$ is contractible as $\Delta P$ is a cone over $\Delta P_{\neq h_0}$ in this case: $\Delta P=C\Delta P_{\neq h_0}$. Particularly, every finite lattice is contractible. Therefore by topological properties of a lattice $L$ we mean those of its proper part.

\begin{theorem}[Homotopy Complementation Formula, Bj\"orner, Walker, \cite{BjornerWelker}]
  Let $L$ be a finite lattice and $z\in L$. Denote $z^\perp=\{ x \in L \mid x\land z=\0,\ x\lor z=\1 \}$. Then:
  \begin{enumerate}
   \item $\overline{L} \setminus z^\perp$ is contractible,
   \item if $z^\perp$ is an antichain (i.e. any two elements of $z^\perp$ are incomparable), then
      $$
        \overline{L} \cong
        \bigvee\limits_{y\in z^\perp} \Sigma\left( \overline{L}_{<y} \ast \overline{L}_{>y} \right),
      $$
    where $X\ast Y$ denotes a join of topological spaces $X$ and $Y$, and $\Sigma(X)$ denotes a suspension over topological space $X$.
  \end{enumerate}
\end{theorem}

We also mention the following well-known corollary of Quillen's Fiber Lemma (see \cite{Shareshian}):

\begin{lemma}
\label{Coatoms}
  Let $L$ be a proper part of some finite lattice $\mathcal{L}$, $M$ be a set of all elements $x\in L$ such that $x=\bigwedge_{c\in C}c$, where $C$ is some subset of maximal elements of $L$. Then $L$ and $M$ are homotopy equivalent.
\end{lemma}

  The last lemma allows us to reduce the complexity of a lattice being examined greatly. For example, it follows that the poset $\Sub A_5$ is homotopy equivalent to a 1-dimensional complex which is connected as well as the original poset. Once can easily check that its reduced Euler characteristics is -60, hence $\Sub A_5$ is homotopy equivalent to a wedge of 60 circles.

  For many minimal simple groups ($L_2(p)$ for prime $p\equiv 3,5\pmod{8}$, $L_2(2^p)$, $L_2(3^p)$ and $Sz(2^p)$ for prime $p$) the situation is similar, however one still needs to use Quillen's fiber lemma to remove a number of elements to get a 1-dimensional complex. 

\begin{corollary}
  Keep the notation of the previous lemma. Let $R$ be a proper part of some sublattice of $\mathcal{L}$ such that $M \subseteq R \subseteq L$. Then $L$ and $R$ are homotopy equivalent.
\end{corollary}
\begin{proof}
  The sets of all non-empty intersections of maximal elements of $L$ and $R$ coincide (with $M$).
\end{proof}

  Unfortunately, Lemma \ref{Coatoms} cannot be used iterative as we cannot delete any new element. Thus, it is naturally to ask a question: is it possible to ``get rid'' of some maximal elements? We were able to show that in a more general case the homotopy type of a poset $P$ can be determined using the topology of its subposets.

\begin{remark}
  Let $P$ be a finite poset, $m \in P$. Then
  $$
    \Delta(P_{<m}\cup P_{>m}) = \Delta P_{<m} * \Delta P_{>m}.
  $$
\end{remark}
\begin{proof}
  Any element of $P_{>m}$ is greater than any element of $P_{<m}$, thus any chain in $P_{<m}\cup P_{>m}$ is a union of some chain in $P_{>m}$ and some chain in $P_{<m}$ (note, that either chain may be empty). But such chains correspond to the simplices of a join of spaces $\Delta P_{<m}$ and $\Delta P_{>m}$.
\end{proof}

\begin{wrapfigure}{r}{5.5cm}
  \begin{center}
    \includegraphics[width=5cm]{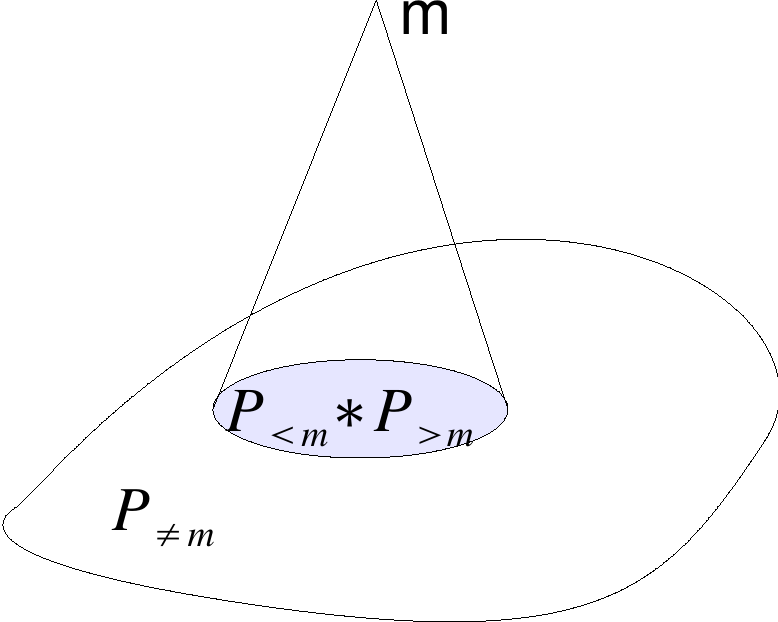}
  \end{center}
  \caption{\small Space $\Delta P$.}
  \label{P_neq_m}
\end{wrapfigure}

\ \begin{lemma}
  Let $P$ be a finite poset, $m \in P$. Let the simplicial complex $\Delta(P_{<m}\cup P_{>m})$ be contractible by $\Delta P_{\neq m}$, then
  $$
    \Delta P\cong \Delta P_{\neq m} \vee \Sigma (\Delta P_{<m}*\Delta P_{>m}). 
  $$
\end{lemma}

\begin{proof}
  A simplicial complex $Q_m=\Delta (P_{<m}\cup P_{>m}\cup \{m\})$ represents a cone with a point $m$ over the base $\Delta P_{<m}*\Delta P_{>m}$. The complex $\Delta P$ is a gluing of $\Delta P_{\neq m}$ and $Q_m$ by $\Delta P_{<m} * \Delta P_{>m}$ (see Figure \ref{P_neq_m}). The homotopy map, contracting $\Delta P_{<m}*\Delta P_{>m}$ by $\Delta P_{\neq m}$ to some point $x$, maps the cone $Q_m$ to a suspension $\Sigma(\Delta P_{<m} * \Delta P_{>m})$. This suspension is glued to $\Delta P_{\neq m}$ by exactly a point $x$.
\end{proof}

\begin{remark}
  If the complex $P_{\neq m}$ is not connected, then the basepoint $x$ of the wedge $x$ must belong to the same component as $\Delta P_{<m}*\Delta P_{>m}$ (because $\Delta P_{<m}*\Delta P_{>m}$ is contractible by $P_{\neq m}$, it must be contained in a single connected component).
\end{remark}

\begin{theorem}
\label{vee_sigma}
  Let $M$ be an antichain of elements of $P$. Assume that the complex $\bigcup_{m\in M}\Delta P_{<m} * \Delta P_{>m}$ is contractible by $\Delta P_{\notin M}$. Then
  \begin{equation}
  \label{vee_M_djoin}
  \Delta P\cong \Delta P_{\notin M}\vee\bigvee_{m\in M}\Sigma (\Delta P_{<m}*\Delta P_{>m}).
  \end{equation}
\end{theorem}

The proof of the last theorem is similar to the proof of the preceeding lemma, so we omit it. Nevertheless, it is worth mentioning that $M$ is an antichain and thus for any $m_1$ and $m_2 \in M$ the cones $Q_{m_1}$ and $Q_{m_2}$ may intersect only by their bases.

\begin{corollary}
  If $\dim\Delta P_{<m}*\Delta P_{>m}=\dim\Delta P_{<m} + \dim\Delta P_{>m} + 1\leq k$ for all $m\in M$ and the complex $\Delta P_{\notin M}$ is $k$-connected (i.e. $\pi_0(\Delta P_{\notin M})=\ldots=\pi_k(\Delta P_{\notin M})=0$), then (\ref{vee_M_djoin}) holds.
\end{corollary}

\begin{lemma}\label{rubanok}
  Let $M$ be a set of some (possibly not all) maximal elements of $P$. Assume that the complex $\bigcup_{m\in M}\Delta P_{<m}$ is contractible by $\Delta P_{\notin M}$, then 
  \begin{equation}
    \label{vee_M} 
    \Delta P\cong \Delta P_{\notin M}\vee\bigvee_{m\in M}\Sigma\Delta P_{<m}.
  \end{equation}
\end{lemma}
\begin{proof}
  Any two maximal elements in a poset $P$ are incomparable. Hence, any subset $M\subseteq P$ consisting of maximal elements is an antichain. This implies that we can use Theorem \ref{vee_sigma}. It is only left to mention that $X*\emptyset=X$ for any topological space $X$.
\end{proof}

Note that if $P$ is a finite lattice and $M$ is a set of some maximal elements of the proper part of $P$, then $P_{\notin M}$ is also a lattice. Thus, Lemma \ref{rubanok} is likely to give good results together with Lemma \ref{Coatoms}: we leave only the intersections of maximal elements, then we delete some maximal elements, then we apply Lemma \ref{Coatoms} again etc. Note, that at any moment we can do the same for minimal elements.

Combining the homotopy methods described above we will be able to determine the exact homotopy type of $\Sub PSL(2,7)$.

\section{Wedge of Spheres of Different Dimensions}

Shareshian made a conjecture (see \cite{ShaHypothesis}) that for any finite group $G$ the simplicial complexes $\Delta\Sub G$ and $\Delta\Co G$ are homotopy equivalent to wedges of spheres of possibly different dimensions. However, it is not even known if homologies in some dimension are torsion-free for arbitrary finite group.

Up to now the attention was focused mainly to minimal simple groups with $\Sub G$ homotopy equivalent to a wedge of circles. We shall consider a minimal simple group $PSL(2,7)$ and prove that the proper part of its subgroup lattice is a wedge of spheres of two different dimensions.

\begin{figure}
  \begin{center}
    \includegraphics[width=16cm]{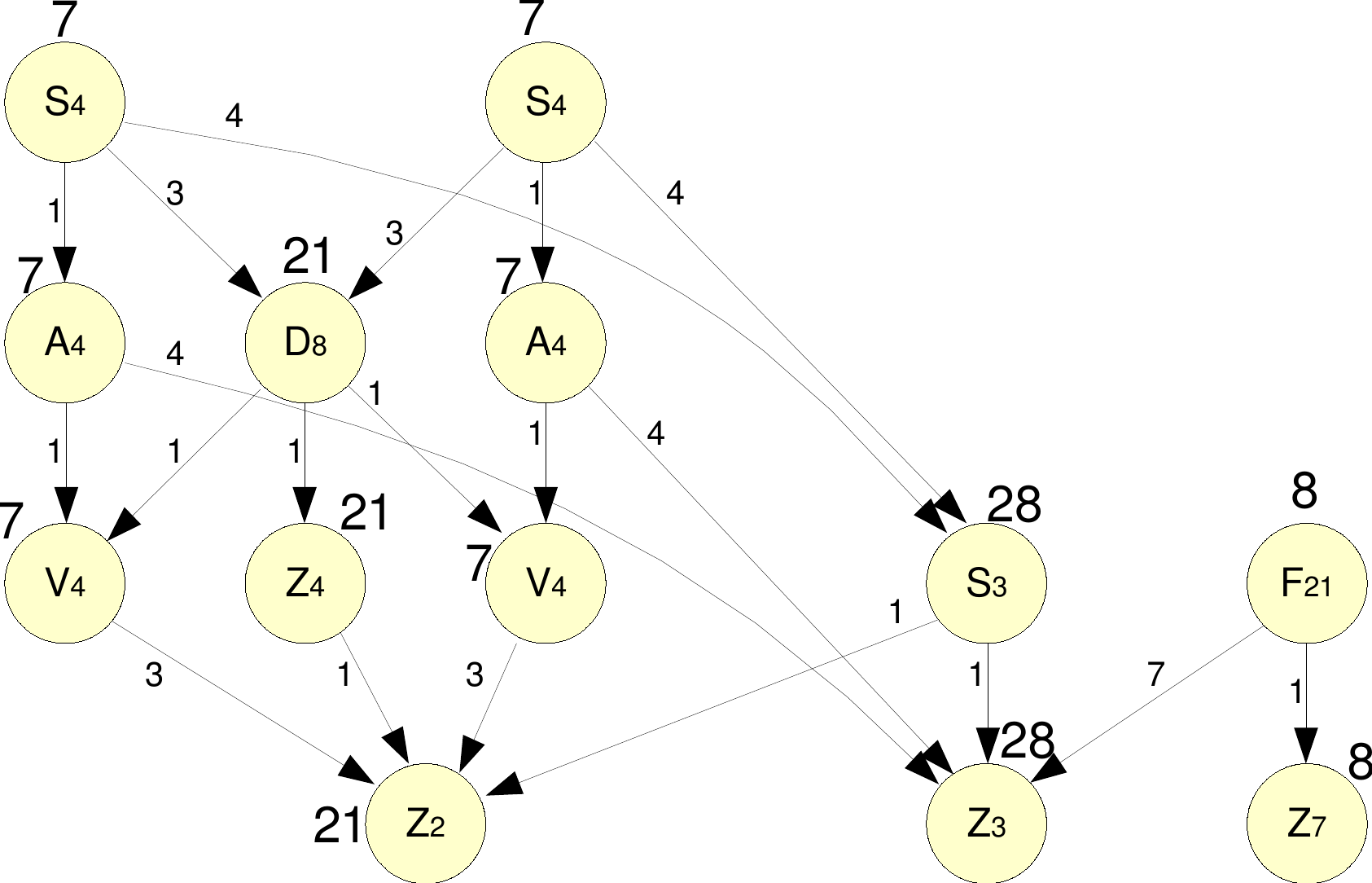}
  \end{center}
  \caption{\small Subgroup lattice of $PSL(2,7)$.}
  \label{PSL27}
\end{figure}

The subgroup lattice of $PSL(2,7)$ is depicted in Figure \ref{PSL27} (see \cite{Dickson}). Each vertex corresponds with a conjugacy class, the cardinality of the class is represented by a number next to a vertex.

If two conjugacy classes are connected by an arrow that means that a group from the upper class contains subgroups from the lower class. The number of such subgroups is equal to a small number next to an arrow.

Figure \ref{PSL27} does not contain all possible arrows: if $H_1 < H_2 < H_3$, then we omit the arrow $H_3\to H_1$ and draw $H_3\to H_2$ and $H_2\to H_1$.

By Lemma \ref{Coatoms} one can consider a smaller poset $Q$ of all nontrivial intersections of maximals subgroups (i.e. conjugacy classes of $A_4$, $\mathbb Z_4$ and $\mathbb Z_7$ are omitted).

Suppose that $\mathcal{M}$ is a set of all subgroups of type $F_{21}$. Each elements of $\mathcal M$ is maximal in $Q$, hence the complex $\Delta Q_{\notin\mathcal M}$ is connected. Note that for any $m\in\mathcal M$ the complex $\Delta Q_{<m}\cong \bigvee\limits^6{S_0}$ (i.e. it consists of $7$ points) and by Lemma \ref{rubanok} and its corollaries we conclude:
$$
  \Delta Q\cong \Delta Q_{\notin\mathcal M}\vee \bigvee^8\Sigma\bigvee^6 S_0=\Delta Q_{\notin\mathcal M}\vee \bigvee^{48} S^1.
$$
Thus, we have isolated a wedge of $48$ circles. Denote $Q_{\notin\mathcal M}$ by $Q^I$. From the fact that $\widetilde \chi(\Sub PSL(2,7))=0$ (see \cite{Hall}), we conclude that $\widetilde \chi(Q^{I})=48$ (as we removed a wedge of $48$ $S^1$). We will show that $\Delta Q^{I}$ is homotopy equivalent to a wedge of $48$ spheres.

For each subgroup in $S_4$ we consider $\Delta Q^I_{<S_4}$. This is a connected 1-dimensional complex with reduced Euler characteristics equal to $-8$ as $\Delta Q'_{<S_4}$ consists of exactly $17$ vertices and $24$ edges. Consequently, $\Delta Q^I_{<S_4}\cong\bigvee^8 S^1$. We delete $6$ vertices  from the right conjugacy class of $S_4$ and show that the resulting poset $Q^{II}$ is contractible. By Theorem \ref{vee_sigma} this means that
$$
  \Delta Q^I\cong \Delta Q^{II}\vee\bigvee^6\Sigma\bigvee^8 S_1 \cong pt\vee\bigvee^{48}S^2 \cong\bigvee^{48}S^2.
$$
\ Again we use Lemma \ref{rubanok} to consider the poset of all nontrivial intersections of maximal subgroups instead of $Q^{II}$ (note that $Q^{II}$ is still a proper part of some lattice, it is depicted in Figure \ref{END}). It is well-known that $PSL_2(7) \cong GL_3(2)$ and thus admits a natural group action on a 3-dimensional vector space over $\mathbb Z_2$.

\begin{wrapfigure}{r}{5.5cm}
  \begin{center}
    \includegraphics[width=5cm]{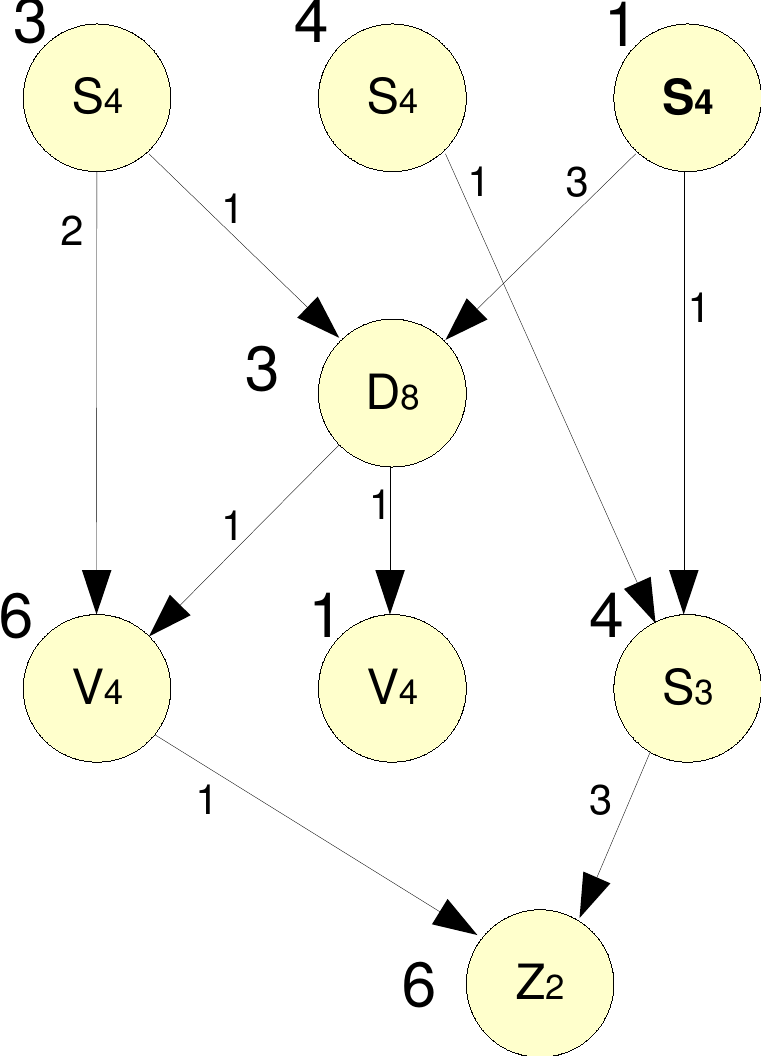}
  \end{center}
  \caption{\small Intersection of maximal elements in $Q^{II}$.}
  \label{END}
\end{wrapfigure}

Without loss of generality it can be assumed that the left conjugacy class of $S_4$ consists of the stabilizers of non-zero vectors. Any pair of non-zero vectors $u\neq v$ can be extended to some basis $(u,v,w)$. The group acts transitively on the set of all bases, hence $St(u,v)$ is exactly the subgroup of operators mapping $(u,v,w)$ to any basis $(u,v,w')$ isomorphic to $V_4$. The stabilizer of any 3 vectors is obviously trivial.

Denote the only element of the right class of $S_4$ by $S$. Suppose that some subgroup $H$ in $Q^{II}$ is not contained in $S$, then it is either $S_4$ (vector stabilizer) or $V_4$ (stabilizer of two vectors). Again without loss of generality we assume that $S$ (as a line stabilizer) consists of all invertible matrices of the form
$$
\begin{pmatrix}
1 & 0 & 0 \\ 
* & * & * \\ 
* & * & *
\end{pmatrix}
$$ 
One can easily check that the set of such matrices contains a nontrivial stabilizer of any two non-zero vectors.

In fact, we proved that any subgroup in $Q^{II}$ intersects with $S$ nontrivially. Consequently, $S^\perp$ is empty and by Homotopy Complementation Formula we conclude that $Q^{II}$ is contractible.

Thus, we proved the following theorem on the homotopy type of the subgroup lattice of $PSL(2,7)$:

\begin{theorem}
  Simplicial complex of the subgroup lattice of $PSL(2,7)$ is homotopy equivalent to a wedge of $48$ circles and $48$ spheres:
  $$
    \Delta\Sub PSL(2,7) \cong \bigvee^{48}S^1 \vee \bigvee^{48}S^2.
  $$
\end{theorem}

\section{Spectral Sequence of Posets}

If we consider groups $PSL(2,7)\times PSL(2,7)\times\ldots\times PSL(2,7)$ or $PSL(2,7)\times(\mathbb Z_2)^n$, it becomes evident that the subgroup lattice of a finite group can be a wedge of spheres of any given number of dimensions and thus its Euler characteristics cannot give the complete information on the homotopy type of this lattice.

We shall use the standard tool of studying the homologies of a topological space~--- a spectral sequence. Let $P$ be a finite poset, consider a natural filtration on $P$:

$$
  \begin{array}{l}
    P^0=P^{\leq 0}=\{h\in P\ |\ h \mbox{ is minimal }\}\subseteq P;\\
    P^{\leq 1}=\{h\in P\ |\ \forall x\mbox{, if }x<h\mbox{, then } x\in P^0\}\subseteq P;\\
    P^{\leq 2}=\{h\in P\ |\ \forall x\mbox{, if }x<h\mbox{, then } x\in P^{\leq 1}\}\subseteq P;\\
    \ldots
  \end{array}
$$

That is, $h\in P^{\leq k}$ exactly when the maximal length of a chain $x_0<x_1<\ldots<x_{k-1}<h$, $x_i\in P$, is $k$. Obviously,
$$
  P^0\subseteq P^{\leq 1}\subseteq P^{\leq 2}\subseteq\ldots\subseteq P^{\leq n}=P.
$$

Thus the maximal length of the chain in $P$ is $n$: $x_0<x_1<\ldots<x_n$ and $\dim \Delta P=n$. We define
$$
  P^k=P^{\leq k}\setminus P^{\leq k-1},\quad\forall k\geq 1.
$$

We call each set $P^k$ a level (namely, the $k$-th level). Note, that if $x<y$ and $y\in P^k$, then $k\neq 0$ and $x\in P^{\leq k-1}$. Now we are ready to proof the main lemma of this section:

\begin{lemma}
\label{nadstrojka1}
  Assume that $h\in P^{k+1}$, $k\geq 0$ and $X_h=\Delta(P^{\leq k}\cup\{h\})$. Then the quotient space $X_h/\Delta P^{\leq k}$ coincides with a suspension over $\Delta P_{<h}$:
  $$
    X_h/\Delta P^{\leq k} = \Sigma\Delta P_{<h}.
  $$
\end{lemma}

\begin{wrapfigure}{r}{5.5cm}
  \begin{center}
    \includegraphics[width=5cm]{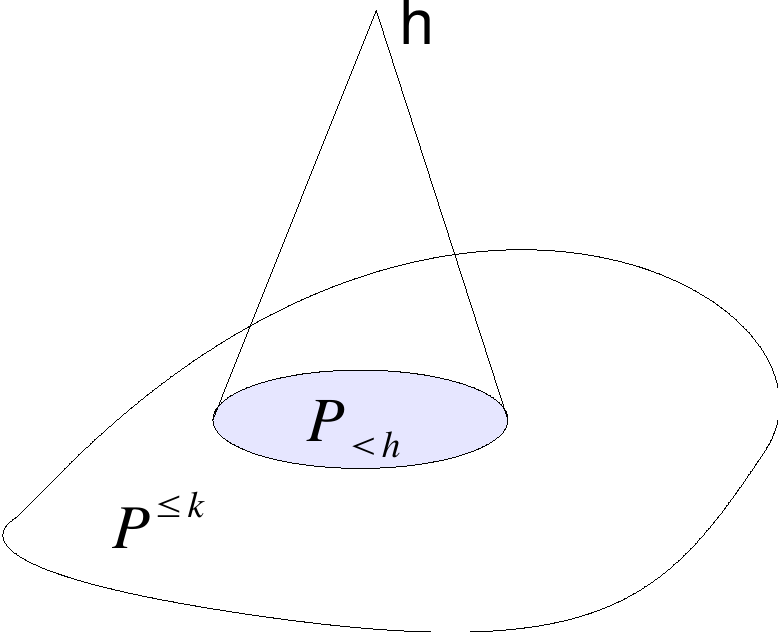}
  \end{center}
  \caption{\small Space $X_h$.}
  \label{X_h}
\end{wrapfigure}

\ \begin{proof}
  Consider a topological space $\Delta P_{\leq h}$. It is obviously a cone over $\Delta P_{<h}\neq\emptyset$. Furthermore, as $P_{<h}\subseteq P^{\leq k}$, the base of the cone $\Delta P_{<h}$ lies in $\Delta P^{\leq k}$.

  Space $X_h$ represents a union of spaces $\Delta P^{\leq k}$ and $\Delta P_{\leq h}$ intersecting by the base of the cone $\Delta P_{<h}$ (see Figure \ref{X_h}):
  $$
    X_h=\Delta P^{\leq k} \cup \Delta P_{\leq h} \quad\mbox{ and }\quad
    \Delta P^{\leq k} \cap \Delta P_{\leq h}=\Delta P_{<h}.
  $$
  Consequently,
  $$
    X_h/\Delta P^{\leq k} = \Delta P_{\leq h}/\Delta P_{<h} = \Sigma P_{<h}.
  $$
\end{proof}

The last lemma shows a strong connection between $\Delta P$ and all of its subspaces $\Delta P_{<h}$. Now we need the following theorem generalizing Lemma \ref{nadstrojka1}:

\begin{theorem}
\label{Buket_nadstrojka}
  Quotient space $\Delta P^{\leq k+1}/\Delta P^{\leq k}$ is a wedge of suspensions over $\Delta P_{<h}$ indexed by all $h\in P^{k+1}$:
  $$
    \Delta P^{\leq k+1}/\Delta P^{\leq k} = \bigvee_{h\in P_{k+1}}\Sigma\Delta P_{<h}.
  $$
\end{theorem}

\begin{wrapfigure}{r}{5.5cm}
  \begin{center}
    \includegraphics[width=5cm]{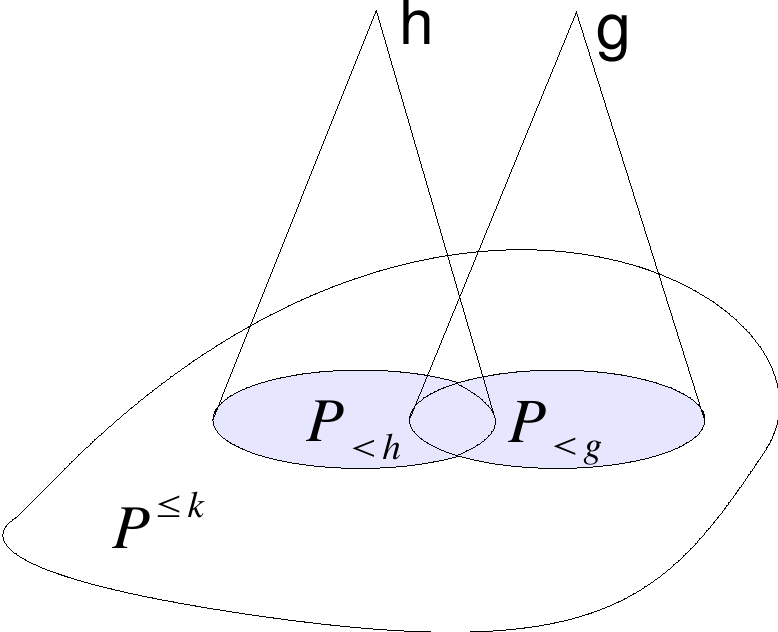}
  \end{center}
  \caption{\small Space $P^{\leq k+1}$.}
  \label{X_hg}
\end{wrapfigure}

\ \begin{proof}
  We have $P^{\leq k+1}=P^{k+1}\cup P^{\leq k}$, futhermore, assume that $x$ and $y$ are in $P^{\leq k+1}$ and $x<y$, then $x\in P^{\leq k}$ and either $y\in P^{k+1}$ or $y\in P^{\leq k}$. Consequently, the space $\Delta P^{\leq k+1}$ is a union of $\Delta P^{\leq k}$  and the cones $\Delta P_{\leq h}$ index by all $h\in P^{k+1}$ (see Figure \ref{X_hg}). Assume $h_1\neq h_2$ and $h_1,h_2\in P^{k+1}$ then the intersection of posets $P_{\leq h_1}\cap P_{\leq h_2}$ coincides with $P_{<h_1}\cap P_{<h_2}\subseteq P^{\leq k}$. Hence, the cones $\Delta P_{\leq h_1}$ and $\Delta P_{\leq h_2}$ may intersect only by their bases, but their bases lie in a space $\Delta P^{\leq k}$. Thus,
  $$
    \Delta P^{\leq k+1}/\Delta P^{\leq k} = 
    \left(\bigcup_{h\in P^{k+1}}X_h\right)/\Delta P^{\leq k} = 
    \bigvee_{h\in P^{k+1}}\left(X_h/\Delta P^{\leq k}\right).
  $$
  By Lemma \ref{nadstrojka1} the last space is a wedge of suspensions over $\Delta P_{<h}$ at the vary of all $h\in P_{k+1}$.
\end{proof}

Note that for any topological spaces $X$ and $Y$ the reduced homologies have the following properties: $\widetilde H_m(\Sigma X)=\widetilde H_{m-1}(X)$ and $\widetilde H_m(X\vee Y)=\widetilde H_m(X)\oplus\widetilde H_m(Y)$ (see \cite{FomenkoFuks}).

The filtration of $P$ defined above induces a natural filtration on $\Delta P$:
$$
  \Delta P^0 \subseteq\Delta P^{\leq 1}\subseteq\Delta P^{\leq 2}\subseteq\ldots\subseteq\Delta P^{\leq n}=\Delta P.
$$

The structure of $\Delta P^{\leq k+1}/\Delta P^{\leq k}$ is determined, so we are ready to write down the spectral sequence for that filtration. First, we describe $E^1$ (we assume that $\widetilde H_m(X)=0$ for all $m<0$ and $X\neq\emptyset$):

\begin{equation}
\label{E_k_l_r}
  \begin{array}{c}
    E^1_{0,0}=H_0(\Delta P^0)=\mathbb Z^{|P_0|},\quad E^1_{0,l}=0\mbox{, for all }l\neq 1\\
    E^1_{k,l}=\widetilde H_{k+l}(\Delta P^{\leq k}/\Delta P^{\leq k-1})=\bigoplus\limits_{h\in P^k}\widetilde H_{k+l-1}(\Delta P_{<h})\mbox{, for all }l\mbox{ and }k\geq 1.\\
  \end{array}
\end{equation}

Now we should say some words about $d^r_{k,l}$ (the arrows): they act from $E^r_{k,l}$ into $E^r_{k-r,l+r-1}$. Thus $E^{r+1}_{k,l}$ is a quotient of the kernel of $d^r_{k,l}$ from $E^r_{k,l}$ by the image of $d^r_{k+r,l-r+1}$ in $E^r_{k,l}$ (in a spectral sequence the image always lies in a kernel):
$$
  E^{r+1}_{k,l}=\ker d^r_{k,l} / \im d^r_{k+r,l-r+1}.
$$

Next, from the inductive construction of $E^{r+1}$ we see that
\begin{enumerate}
  \item If $E^1_{k,l}=0$ for some $k$ and $l$, then for all $r\geq 1$ we have $E^r_{k,l}=0$.
  \item For all $k$, $l$ and $r\geq 0$ we have $\dim E^{r+1}_{k,l}\leq \dim E^r_{k,l}\leq\ldots\leq \dim E^1_{k,l}$ (by $\dim$ we will mean the torsion-free rank of an abelian group).
  \item All the arrows $d^r_{k,l}$ starting from the diagonal $k+l=m$ point to cells on a diagonal $k+l=m-1$ regardless of $r$, and vice versa, the arrows ending on a diagonal $k+l=m$ start from the cells on a diagonal $k+l=m+1$ regardless of $r$. Thus we have
  $$
    \sum_{k+l=m}\dim E^r_{k,l}\geq\sum_{k+l=m} \dim E^1_{k,m} - \sum_{k+l=m+1} \dim E^1_{k,m} - \sum_{k+l=m-1} \dim E^1_{k,m}.
  $$
\end{enumerate}

From the dimension argument we conclude that spectral sequence stabilizes on $(n+1)$th step: $E^{n+1}=E^{n+2}=\ldots=E^\infty$. $E^\infty$ can be used to determine the homologies of $\Delta P^{\leq n}=\Delta P$ from the diagonal $k+l=m$ (for fixed $m$): denote all nontrivial groups on this diagonal starting from the top as $G_1$, $G_2$, $\ldots$, $G_s$. Then $G_1$ is a subgroup in $H_m(\Delta P)$; $G_2$ is a subgroup in $H_m(\Delta P)/G_1$, $G_3$ is a subgroup in $(H_m(\Delta P)/G_1)/G_2$ etc. The last group $G_s$ coincides with the considered quotient group (see \cite{FomenkoFuks}).

Using the fact that every diagonal $k+l=m$ in $E^1$ contains sums of groups of type $\widetilde H_m(\Delta P_{<h})$ (except for the 0th column $E^1_{0,l}$, where the only non-zero cell is $E^1_{0,0}$), it is possible to formulate the following theorems (we assume that $\widetilde H_{-1}(\emptyset)=\mathbb Z$ and $\widetilde H_m(\emptyset)=0$ for any $m\neq -1$ ):

\begin{theorem}
\label{H_m_leq}
  For any $m\geq 0$ we have the following estimate for the Betti numbers of the poset $P$ (i.e. the ranks of its homology groups):
  $$
    \dim H_m(\Delta P)\leq \sum_{h\in P}\dim\widetilde H_{m-1}(\Delta P_{<h}).
  $$
\end{theorem}

\begin{proof}
  Fix $m\geq 1$. By the properties of the spectral sequence given above we conclude:

  \begin{enumerate}
    \item $\dim H_m(\Delta P)$ is a sum of dimensions $\dim E^\infty_{k,l}$ for all $E^\infty_{k,l}$ on the diagonal $k+l=m$.
    \item The dimension of $E^r_{k,l}$ cannot grow with the increase of $r$.
  \end{enumerate}

  It easily follows that
  $$
    \dim H_m(\Delta P)\leq\sum_{k+l=m}\dim E^1_{k,l}=\sum_{h\in P}\widetilde H_{m-1}(\Delta P_{<h}).
  $$

  If $m=0$, then the diagonal $k+l=0$ can contain more than one non-zero element: $E^r_{0,0}\subseteq E^1_{0,0}=\mathbb Z^{|P_0|}$, and $|P_0|=\sum_{h\in P^0}\widetilde H_{-1}(\Delta P_{<h})=\sum_{h\in P}\widetilde H_{-1}(\Delta P_{<h})$.
\end{proof}

The spectral sequence constructed above may be just as well applied to prove the absence of torsion:

\begin{theorem}
\label{homology_0}
  Suppose that for some fixed $m\geq 0$ and for any $h\in P$ we have $\widetilde H_m(\Delta P_{<h})=0$, then $H_{m+1}(\Delta P)=0$. Furthermore, if for all $h\in P$ the homologies $\widetilde H_{m-1}(\Delta P_{<h})$ are torsion-free, then $H_m(\Delta P)$ are also torsion-free.
\end{theorem}

\begin{proof}
  If for some $m\geq 0$ the homologies $\widetilde H_m(\Delta P_{<h})=0$ for all $h\in P$, then all the elements on a diagonal $k+l=m+1$ in the spectral sequence are zeroes for all $r\geq 1$ and thus $H_{m+1}(\Delta P)=0$.

  Moreover, for all $r \geq 1$ the diagonal $k+l=m+1$ in the spectral sequence contains only zeroes, hence for all $r\geq 1$ the image of the map $d^r_{k,l}$ for $k+l=m+1$ is equal to $0$. Thus, in the inductive construction of $E^r_{k,l}$ for $k+l=m$ we will take quotients by $0$. This means that for $m\neq 0$ each group $E^\infty_{k,l}$ on a diagonal $k+l=m$ is a subgroup of $E^1_{k,l}$.

  If $E^1_{k,l}=\bigoplus_{h\in P^k}\widetilde H_{m-1}(\Delta P_{<h})$ is torsion-free, then $E^r_{k,l}\subseteq E^1_{k,l}$ is torsion-free, hence $H_m(\Delta P)$ is also torsion-free.
\end{proof}

Theorem \ref{homology_0} is very useful for dealing with torsion in higher non-vanishing homologies of a subgroup lattice or a coset lattice of a finite group $G$: it is sufficient to prove that all the subgroups of a given group $G$ have torsion-free higher homologies and that higher non-vanishing homologies of $G$ have dimension at least one more than that of any of its subgroups.

\begin{corollary}
\label{sled_Hdim_leq}
  Suppose that there exist such $m_0$ that for all $m\geq m_0$ and for all $h\in P$ the homologies $\widetilde H_m(P_{<h})$ vanish, then for all $m\geq m_0+1$ the homologies $H_{m}(\Delta P)$ also vanish:
    $$
      \exists m_0\geq 0:\ \forall m\geq m_0\ \forall h\in P\ \ \widetilde H_m(\Delta P_{<h})=0\quad\Rightarrow\quad
      \forall m\geq m_0+1\ \ H_m(\Delta P)=0
    $$
\end{corollary}

\begin{theorem}
\label{H_m_geq}
  The following lower estimate holds for the Betti Numbers of a poset $P$: if $m\geq 1$, then
  $$
    \dim H_{m+1}(\Delta P)\geq
    \sum_{h\in P}\bigl(\dim\widetilde H_m(\Delta P_{<h})-\dim\widetilde H_{m-1}(\Delta P_{<h})-\dim\widetilde H_{m+1}(\Delta P_{<h})\bigr),
  $$
  if $m=0$, then
  $$
    \dim H_1(\Delta P)\geq
    \sum_{h\in P}\bigl(\dim\widetilde H_0(\Delta P_{<h})-\dim\widetilde H_1(\Delta P_{<h})-|P_0|),
  $$
  and obviously
  $$
    \dim H_0(\Delta P)\geq |P_0|-\sum_{h\in P}\dim\widetilde H_0(\Delta P_{<h}).
  $$
\end{theorem}

\begin{proof}
  The estimates can be easily deduced from the fact the sum of all dimensions $E^\infty_{k,l}$ on a diagonal $k+l=m$ with respect to the total dimension of $E^1_{k,l}$ on the same diagonal cannot decrease by more than the sum of all dimensions of $E^1_{k,l}$ on the diagonals $k+l=m+1$ and $k+l=m-1$.
\end{proof}

Unfortunately, the right part of these expressions is often negative.

But if we know some homology groups of $\Delta P$ (for example if $\Delta P$ is connected, then $H_0(\Delta P)=\mathbb Z$), then the method demonstrated above allows us to obtain sharper estimates for the Betti numbers.

It is natural that this technic is useful in the case when Euler characteristics does not contain the full information on the homotopy type of $\Delta P$. That is the case when $\Delta P$ is not homotopy equivalent to a wedge of equidimensional spheres.

We also note that to use Theorems \ref{H_m_leq}, \ref{homology_0} and \ref{H_m_geq} one needs to know only the homologies of complexes $P_{<h}$, not the way they are linkes with each other. Thus for the subgroup lattice and the coset lattice of some finite group $G$ it is sufficient to know only the types of its subgroups and their number, but not the exact structure of the whole lattice.

\section{Decreasing Posets}

\begin{definition}
  For a topological space $X$ define its homology dimension as the maximal dimension of its non-vanishing reduced homologies:
  $$
    \Hdim X=\max\{m: \widetilde H_m(X)\neq 0\}.
  $$
  If all the reduced homologies $X$ are vanishing or $X=\emptyset$, we assume that $\Hdim X =-1$.
\end{definition}

For example, if $X$ is a wedge of spheres of possibly different dimensions, then $\Hdim X$ is the maximal dimension of spheres in the wedge.

The Cell Homology Theorem (see \cite{FomenkoFuks}) states that homology dimension of any simplicial complex does not exceed its ordinary dimension:
$$
  \Hdim\Delta\leq\dim\Delta.
$$

Corollary \ref{sled_Hdim_leq} can be reformulated in the new notation as
\begin{equation}
\label{Hdim_leq}
  \Hdim P\leq 1 + \max_{h\in P}\Hdim P_{<h}.
\end{equation}

Now we define the concept of decreasing poset and decreasing level inductively. We start from the bottom: the empty set $\emptyset$ is not decreasing. Let $P$ be some poset and the property of being decreasing is defined for subposets $P_{<h}$ for all $h\in P$. We will define this property for $P$.

As in the previous section we divide the poset into levels: $P^0$, $P^1$, $\ldots$, $P^n$. We say that the level $P^k$ is decreasing provided the poset $P_{<h}$ is decreasing for all $h\in P^k$. Let $s(P)$ be the number of decreasing levels in $P$.

We say that the poset $P$ is decreasing if 
$$
  \Hdim P\leq \dim P - s(P) - 1.
$$
Thus, we have defined the concept of being ``decreasing'' for 1-dimensional posets, then for 2-dimensional etc. For example, any contractible poset (including a single point) is decreasing, but any 0-dimensional poset except for a point is not.

\begin{lemma}
\label{ponij_basa}
  Suppose that $\dim\Delta P=n\geq 0$ and the level $P^{n}$ is decreasing. Then $\Hdim \Delta P\leq n-1$.
\end{lemma}

\begin{proof}
  If $h\in P^{n}$, then $\dim\Delta P_{<h}=n-1$. The poset $P_{<h}$ is decreasing, so its homology dimension cannot be maximal. Hence, $\Hdim\Delta P_{<h}\leq n-2$.

  Now suppose that $h\notin P^{n}$, then $\Hdim\Delta P_{<h}\leq\dim\Delta P_{<h}\leq n-2$. From (\ref{Hdim_leq}) we conclude that $\Hdim \Delta P\leq n-1$.
\end{proof}

\begin{lemma}
\label{ponij_basa_basa}
  Suppose that $\dim\Delta P=n\geq 0$ and the level $P^k$ is decreasing for some $k\geq 0$. Then $\Hdim\Delta P\leq n-1$.
\end{lemma}

\begin{proof}
  If $k=n$, then the statement is equivalent to Lemma \ref{ponij_basa}. Suppose $k<n$. Again by Lemma \ref{ponij_basa} the top level of $P_{<h}$ is decreasing for all $h\in P^{k+1}$. Hence, $\Hdim\Delta P_{<h}\leq\dim\Delta P_{<h}-1=k-1$.

  We use the induction: if for some $l\geq k+1$ and for all $h\in P^{\leq l}$ the homology dimension $\Hdim\Delta P_{<h}$ does not exceed $l-2$, then for all $h$ from the level $l+1$ we have:
  $$
    \forall h\in P^{l+1}\quad \Hdim P_{<h}\leq 1 + \max_{h\in P^{\leq l}}\Hdim P_{<h}\leq l-1.
  $$

  Thus, when we ``move'' to the next level, the homology dimension $\Hdim\Delta P_{<h}$ cannot increase by more than 1.

  Inductively considering the levels from the bottom we reach $P$ itself.
\end{proof}

\begin{theorem}
\label{thm_ponij_s_levels}
  If the poset $P$ containt exactly $s(P)$ decreasing levels, than $\Hdim\Delta P\leq n-s(P)$. Moreover, $P$ is decreasing exactly when $\Hdim\Delta P\leq n-s(P)-1$.
\end{theorem}

\begin{proof}
  Moving from $P^{k}$ to the next level we see that homology dimension of posets $P_{<h}$ does not increase, if $P^{k+1}$ is decreasing, and possibly increases by 1 otherwise. As $P$ contains exactly $s(P)$ decreasing levels, we conclude
  $$
    \max_{h\in P}\Hdim P_{<h}\leq n-s(P)-1.
  $$
\end{proof}

As we shall see later, this theorem has a useful application to the group lattices.

Thus, the existence of decreasing subposets $P_{<h}$ allows one to bound the homology dimension of $\Delta P$. It is natural to wish for some tools able to determine, whether a given poset is decreasing or not. So we consider a fixed level $P^k$ in a poset $P$. Let $s_k(P)$ be the number of decreasing posets in $P$ below $P^k$.

\begin{theorem}
  Suppose there exists some non-decreasing level $P^k$ such that for any $h\in P^k$ one of the following is true:
  \begin{enumerate}
    \item The poset $P_{<h}$ is decreasing;
    \item The poset $P_{<h}$ contains at least $s_k(P)+1$ decreasing levels.
  \end{enumerate}
  Then the poset $P$ is decreasing.
\end{theorem}

\begin{proof}
  As for each $h\in P^k$ a complex $P_{<h}$ contains at least $s_k(P)$ decreasing levels and $\dim P_{<h}=k-1$, Theorem \ref{thm_ponij_s_levels} states that $\Hdim P_{<h}$ does not exceed $\dim P_{<h}-s_k(P)=k-s_k(P)-1$. If $P_{<h}$ is decreasing, then $\Hdim P_{<h}\leq k-s_k(P)-2$. In the opposite case by the conditions of the theorem it contains at least $s_k(P)+1$ decreasing levels and again we conclude that $\Hdim P_{<h}\leq k-s_k(P)-2$.

  Thus, $\Hdim P_{<h}\leq k-s_k(P)-2$ for all $h\in P^k$, while the level $P^k$ is non-decreasing. It follows that there was no dimension increase between levels $P^{k-1}$ and $P^k$ and therefore $\Hdim P\leq n-s(P)-1$.
\end{proof}

\section{The Case of Group Lattices}

Let $G$ be a finite group. We use the following notation: $\Sub G=\{H \mid 1<H<G \}$ is a proper part of the subgroup lattice (ordered by inclusion), $\Co G= \{ xH \mid H<G,\ x\in G \}$ is a proper part of the coset lattice (by inclusion), $\Sco G = \Co G\setminus \{g \in G\}$.

Note that for $H \leq G$ we have $\Co G_{<H}=\Co H$. This fact is really handy for performing calculations in $\Co G$ as this poset contains lots of isomorphic fibers $\Co G_{<gH}$: if $g_1H_1$ and $g_2H_2$ are cosets of isomorphic subgroups $H_1$ and $H_2$ respectively, then $\Co G_{<g_1H_1}\cong \Co G_{<g_2H_2}$. The same holds for $\Sub G$ and $\Sco G$.

Now we reformulate the results proved by the spectral sequence method using the language of group lattices.

\begin{theorem}
\label{group_H_m_leq}
  Betti numbers of $\Sub G$, $\Co G$ or $\Sco G$ for all $m\geq 0$ can be majorized by Betti numbers of all nontrivial proper subgroups in the following way:
  $$
    \begin{array}{rl}
      \dim H_{m}(\Sub G) \leq &\sum\limits_{1<H<G}\dim\widetilde H_{m-1}(\Sub H),\\
      \dim H_{m}(\Co G) \leq &\sum\limits_{1\leq H<G}|G:H|\dim\widetilde H_{m-1}(\Co H),\\
      \dim H_{m}(\Sco G) \leq &\sum\limits_{1<H<G}|G:H|\dim\widetilde H_{m-1}(\Sco H).\\
    \end{array}
  $$
\end{theorem}

\begin{proof}
  For $\Sub G$ the statement is just a reformulation of Theorem \ref{H_m_leq}. For $\Co G$ and $\Sco G$ we need to mention that subposets $\Co G_{<H}$ and  $\Co G_{<gH}$ are isomorphic for every subgroup $H$ and every coset $gH$ of $H$ and the total number of cosets of $H$ is its index $|G:H|$.
\end{proof}

\begin{theorem}
\label{group_homology_0}
  Let $\mathcal P G$ be one of the posets $\Sub G$, $\Co G$ or $\Sco G$. If for some $m\geq 0$ and for all subgroups $1<H<G$ we have $\widetilde H_m(\Delta\mathcal P H) = 0$, then $H_{m+1}(\Delta\mathcal P G) = 0$. Furthermore, assume that for all $1<H<G$ homologies $\widetilde H_{m-1}(\Delta\mathcal P H)$ are torsion-free, then $H_m(\Delta\mathcal P G)$ are torsion-free.
\end{theorem}
\begin{proof}
  The statement is a reformulation of \ref{homology_0}.
\end{proof}

\begin{corollary}
  Let $\mathcal{P}G$ be one of the posets $\Sub G$, $\Co G$ or $\Sco G$. If the maximal dimension of higher non-vanishing homologies of complexes $\Delta\mathcal P H$ at the vary of all proper subgroups $1<H<G$ is $m_0$ (if there are no such $H$, i.e. $G\cong \mathbb{Z}_p$, we assume $m_0=-1$), then the dimension of higher non-vanishing homologies of $\Delta\mathcal{P} G$ does not exceed $m_0+1$:

  $$
    \begin{array}{l}
      \Hdim \Sub G\leq 1 + \max\limits_{1<H<G}\Hdim\Sub H,\\
      \Hdim \Co G\leq 1 + \max\limits_{1<H<G}\Hdim\Co H,\\
      \Hdim \Sco G\leq 1 + \max\limits_{1<H<G}\Hdim\Sco H.\\
    \end{array}
  $$
\end{corollary}

\begin{corollary}
\label{group_no_tors}
  If for any proper subgroup $1<H<G$ higher non-vanishing homologies $H_{m_0}(\Delta\mathcal P H)$ are torsion-free and the poset $\mathcal P G$ is decreasing (i.e. the higher non-zero homologies of $\Delta\mathcal P G$ are exactly of dimension $m_0+1$), then higher homologies of $\Delta\mathcal P G$ are torsion-free.
\end{corollary}

Consider a coset lattice of $PSL(2,7)$. The simplicial complex $\Delta\Co PSL(2,7)$ has dimension $4$ (the longest chain is $S_4 > A_4 > V_4 > \mathbb Z_2 > 1$). Howerver, coset posets of any subgroup: $\mathbb Z_2$, $\mathbb Z_3$, $\mathbb Z_4$, $\mathbb Z_7$, $V_4$, $D_8$, $A_4$, $F_{21}$, $S_3$ and $S_4$ have non-zero homologies either in dimension $0$ ($\mathbb Z_2$, $\mathbb Z_3$, $\mathbb Z_4$, $\mathbb Z_7$), or in dimension $1$ ($V_4$, $A_4$, $D_8$, $F_{21}$ and $S_3$), or in dimension $2$ ($S_4$). As all proper subgroups of $PSL(2,7)$ are solvable, their homologies are torsion-free (see \cite{KraThev}). Applying Thorem \ref{group_homology_0} and Corollary \ref{group_no_tors} we obtain
$$
  \begin{array}{l}
    \widetilde H_4(\Delta\Co PSL(2,7))=0,\\
    \widetilde H_3(\Delta\Co PSL(2,7))\mbox{ is torsion-free.}\\
  \end{array}
$$
Obviously, $\Delta\Co PSL(2,7)$ is connected. Moreover, $\Co PSL(2,7)$ proved to be simply connected (see \cite{Ramras}). Therefore by Theorems \ref{group_H_m_leq} and \ref{H_m_geq} we immediately obtain some estimates for homology ranks of $\Co PSL(2,7)$:
$$
  \begin{array}{lll}
    &\widetilde H_0(\Co PSL(2,7))&=0;\\
    &\widetilde H_1(\Co PSL(2,7))&=0;\\
    \chi(\Co PSL(2,7))\leq&\widetilde H_2(\Co PSL(2,7))&\leq 14616;\\
    &\widetilde H_3(\Co PSL(2,7))&\leq 11760;\\
    &\widetilde H_4(\Co PSL(2,7))&=0.\\
  \end{array}
$$

Now we shall find a connection between homologies of posets $\Sub G$, $\Co G$ and $\Sco G$ for any finite $G$. Consider the opposite filtration $\Co G$ (starting from the maximal cosets and ending by single elements): assume that $\dim\Delta\Co G=n+1$, then $\Co G^{n+1}=G$, as for any $g_0\in G$ the subposet of cosets containing $g_0$ is isomorphic to $\Sub G$, hence, $\dim\Delta\Sub G=n$.

\begin{theorem}
\label{LCS_G}
  Consider a poset $\mathcal S G=\Co G\setminus\Co G^{n+1}\subseteq\Co G$. Assume that for some $m$ we have $\widetilde H_{m}(\Delta\mathcal S G)=0$, then the homology groups $\widetilde H_{m}(\Delta\Co G)$ can be embedded into homology groups $\widetilde H_{m-1}(\Delta\Sub G)^{|G|}$ and there exists a surjection $\widetilde H_{m+1}(\Delta\Co G)\to \widetilde H_m(\Delta\Sub G)^{|G|}$.
\end{theorem}

\begin{proof}
  The easiest way to prove this theorem is to use an exact sequence of a pair (see \cite{FomenkoFuks}): consider a topological pair $(\Delta\Co G,\Delta\mathcal S G)$. By Theorem \ref{Buket_nadstrojka} we have
  $$
    \Delta\Co G/\Delta\mathcal S G =
    \Delta\Co G^{\leq n+1}/\Delta\Co G^{\leq n} =
    \bigvee_{g\in G}\Sigma \Co G_{>\{g\}} =
    \bigvee_{g\in G}\Sigma \Sub G =
    \bigvee^{|G|}\Sigma \Sub G.
  $$

  The exact sequence for the pair $(\Delta\Co G,\Delta\mathcal S G)$ is the following:
  $$
    \ldots\to
    \widetilde H_{m+1}(\Co G)\to
    \widetilde H_{m+1}(\bigvee^{|G|}\Sigma \Sub G)\to
    \widetilde H_{m}(\mathcal S G)\to
    \widetilde H_m(\Co G)\to
    \widetilde H_m(\bigvee^{|G|}\Sigma \Sub G)\to
    \ldots
  $$

  If $\widetilde H_{m}(\mathcal S G)=0$, then we get an injection $\widetilde H_{m}(\Delta\Co G)\to\widetilde H_{m-1}(\Delta\Sub G)^{|G|}$ and a surjection $\widetilde H_{m+1}(\Delta\Co G)\to \widetilde H_m(\Delta\Sub G)^{|G|}$.
\end{proof}

\begin{corollary}
\label{CG_LG}
  Assume that $\dim\Delta\Sub G=n\geq 0$. Then $\dim\Delta\Co G=n+1$ and $H_{n+1}(\Delta\Co G)$ are embedded into $\widetilde H_n(\Delta\Sub G)^{|G|}$. Particularly, if homologies $\widetilde H_n(\Delta\Sub G)=0$ (or are torsion-free), then $H_{n+1}(\Delta\Co G)=0$ (or are torsion-free, respectively).
\end{corollary}

\begin{proof}
  As $\dim\mathcal S G=n$, it is sufficient to use Theorem \ref{LCS_G} for the case $m=n+1$.
\end{proof}

\begin{corollary}
\label{CG_LG_dim}
  The following estimate for the Betti numbers holds:
  $$
    \dim\widetilde H_{n+1}(\Delta\Co G)\leq |G|\dim\widetilde H_n(\Delta\Sub G).
  $$
\end{corollary}

Now we apply the concept of decreasing posets to the case of subgroup lattices. It is important to mention that if the hypothesis that $\Sub G$, $\Co G$ and $\Sco G$ are all homotopy equivalent to wedges of spheres of possibly different dimensions (see \cite{ShaHypothesis}), then their homology dimensions coincide with maximal dimensions of spheres in the wedges.

We should say some facts about the structure of levels in group lattices: for any $k\geq 0$ levels $\Sub G^k$, $\Sco G^k$ and $\Co G^{k+1}$ in the posets $\Sub G$, $\Co G$ and $\Sco G$ respectively contain the same subgroups. A level in each poset $\mathcal P G$ is decreasing provided the corresponding posets $\mathcal P H$ are decreasing for all subgroups on this level.

\begin{theorem}
  Let $\mathcal P G$ be one of the posets $\Sub G$, $\Co G$ or $\Sco G$ and $\dim\Delta\mathcal P G=n$. Assume that $\mathcal P G$ contains exactly $s(\mathcal P G)$ decreasing levels. Then $\Hdim\Delta\mathcal P G\leq n-s(\mathcal P G)$. Moreover, $\mathcal P G$ is decreasing if and only if $\Hdim\Delta\Sub G\leq n-s(\mathcal P G)-1$.
\end{theorem}

\begin{proof}
  A direct reformulation of Theorem \ref{thm_ponij_s_levels}.
\end{proof}

Corollary \ref{group_no_tors} can be reformulated in the following way:

\begin{corollary}
  Suppose that $G$ is a finite group, the higher homologies of $\mathcal P H$ are torsion-free for all the proper subgroups $H$ of $G$ and the poset $\mathcal P G$ is not decreasing. Then the higher homologies of $\mathcal P G$ are torsion-free.
\end{corollary}

Posets $\Sub G$ and $\Co G$ proved to be connected in the sense of being ``decreasing'':

\begin{lemma}
  Suppose neither $\Sub G$ nor $\Co G$ contains a decreasing level. If $\Sub G$ is decreasing, then $\Co G$ is decreasing.
\end{lemma}
\begin{proof}
  An obvious corollary of \ref{CG_LG}.
\end{proof}

\section{Suzuki Groups}

Consider a group $G=\Sz(2^{p^k})$, where $p$ is prime and $k\geq 2$ (for $k=1$ the homotopy type of $\Sub G$ was completely determined in \cite{Shareshian}). Then all subgroups of $G$ are either solvable or isomorphic to $\Sz(2^{p^l})$ for some $l<k$. Obviously, the number of subgroups conjugated to a given subgroup $H=\Sz(2^{p^l})$ in $G$ coincides its index $|G:H|$, as $H$ is self-normalizing. Furthermore, all subgroups isomorphic to $\Sz(2^{p^l})$ are contained in a single conjugacy class in $G$ for all $l<k$ (see \cite{Suzuki}). Thus we have
$$
  |\Sz(2^{p^k}):\Sz(2^{p^l})|=|\Sz(2^{p^k}):\Sz(2^{p^{l+1}})||\Sz(2^{p^{l+1}}):\Sz(2^{p^l})|\mbox{ for }l<k.
$$

In particular, every subgroup $\Sz(2^{p^l})$ is contained in a single subgroup $\Sz(2^{p^{l+1}})$, hence, using Quillen's Fiber Lemma we can drop all subgroups isomorphic to $\Sz(2^{p^l})$ for $l<k-1$ without changing the homotopy type of $\Sub G$. Let $R$ be a poset of all solvable subgroups of $G$ and $S$ be a set of subgroups of the type $G'=\Sz(2^{p^{k-1}})$. Then
$$
  \Delta\Sub G\cong\Delta(S\cup R).
$$

Moreover, every subgroup in $S$ is maximal in both $\Sub G$ and $S\cup R\subseteq\Sub G$. By Lemma \ref{nadstrojka1} 
$$
  \Delta\Sub G/\Delta R\cong\Delta(S\cup R)/\Delta R\cong\bigvee\limits^{|G:G'|}\Sigma\Delta((S \cup R)_{<G'}).
$$

In fact, Shareshian proved (see \cite{Shareshian}) that $\Delta R$ is homotopy equivalent to a wedge of circles: $\Delta R\cong\bigvee_{\widetilde\chi(G)}S^1$. Thus, we are ready to use the spectral sequence method: consider a filtration of a complex $\Delta(S\cup R)$: $\Delta R\subseteq\Delta(S\cup R)\cong\Sub G$. In the resulting spectral sequence $E^1$ will contain only two non-zero cells (except for $E^1_{0,0}=\mathbb Z$):
$$
  E^1_{0,1}=\mathbb Z^{|G|};\quad E^1_{1,1}= (\mathbb Z^{|G'|})^{|G:G'|}=\mathbb Z^{|G|}.
$$

Thus we easily deduce the following statements:
\begin{enumerate}
  \item $\Hdim\Sub\Sz(2^{p^k})\leq 2$ for $k\geq 2$. If $k=1$, then obviously $\Hdim\Sub\Sz(2^{p})=1$.
  \item The reduced homologies of $\Sub\Sz(2^{p^k})$ for $k\geq 2$ have the following structure:
  $$
    \begin{array}{l}
      \widetilde H_2(\Sub\Sz(2^{p^k}))=\mathbb Z^s,\\
      \widetilde H_1(\Sub\Sz(2^{p^k}))=\mathbb Z^s\oplus T,\\
    \end{array}
  $$
  where $0\leq s\leq |G|$ and $T$ is a finite abelian group (torsion part).
\end{enumerate}

Consider a group $G=\Sz(2^{pq})$, where $p$ and $q$ are different primes. Let $R\subseteq\Sub G$ be again a set of solvable subgroups of $G$, and $S$ be a set of all simple subgroups of $G$, i.e. a union of two conjugacy classes of $G_p=\Sz(2^{p})$ and $G_q=\Sz(2^{q})$. Then $\Sub G=R\cup S$ and Lemma \ref{nadstrojka1} yields
$$
  \Delta\Sub G/\Delta R=\Delta(S\cup R)/\Delta R\cong
  \bigvee\limits^{|G:G_p|}\Sigma\Delta\Sub G_p \vee
  \bigvee\limits^{|G:G_q|}\Sigma\Delta\Sub G_q.
$$
Thus, a spectral sequence constructed using the same filtration of $\Delta R\subseteq\Delta G$ will contain only two non-zero cells (except for uninteresting $E^1_{0,0}=\mathbb Z$):
$$
  E^1_{0,1}=\mathbb Z^{|G|};\quad E^1_{1,1}=\mathbb Z^{2|G|}
$$

This means that reduced homologies $\Sub\Sz(2^{pq})$, $p$ and $q$ being prime, have the following structure:
$$
  \begin{array}{l}
    \widetilde H_2(\Sub\Sz(2^{p^k}))=\mathbb Z^{|G|+s},\\
    \widetilde H_1(\Sub\Sz(2^{p^k}))=\mathbb Z^s\oplus T,\\
  \end{array}
$$
where $0\leq s\leq |G|$ and $T$ is a torsion part.

For any Suzuki group, the estimates obtained by the same method, are notably less precise:
$$
  \begin{array}{l}
    \dim\widetilde H_1(\Sub\Sz(2^r))\leq |G|,\\
    \dim\widetilde H_{k+1}(\Sub\Sz(2^r))\leq\sum\limits_{r'|r}\dim\widetilde H_k(\Sub\Sz(2^{r'}))\mbox{ for all }k\geq 1.
  \end{array}
$$


\begin{thebibliography}{99}
  \bibitem{Brown} Brown K.S. The coset poset and probabilistic zeta function of a finite group // J. Algebra. -- 2000. -- Vol. 225, №2. -- P. 989-1012.
  \bibitem{BjornerWelker} Bj\"orner A., Walker J.W. A homotopy complementation formula for partially ordered sets // European J. Combin. -- 1983. -- Vol. 4. -- P.11-19.
  \bibitem{Dickson} Dickson L.E. Linear Groups with an Exposition of the Galois Theory. -- New York: Dover, 1984.
  \bibitem{Hall} Hall P. The Eulerian functions of a group // Quart. J.Math. -- 1936. -- Vol. 7. -- P.134-151.
  \bibitem{KraThev} Kratzer C., Thevenaz J. Type d'homotopie des treillis et treillis des sous-groupes d'un groupe fini // Comment. Math. Helvetici. -- 1985. -- Vol. 60. -- P.85-106.
  \bibitem{Shareshian} Shareshian J. On the shellability of the order complex of the subgroup lattice of a finite group // Trans. Amer. Math. Soc. -- 2001. -- Vol. 353, №7. -- P.2689–2703.
  \bibitem{ShaHypothesis} Shareshian J. Topology of order complexes of intervals in subgroup lattices // J. Algebra. -- 2003. -- Vol. 268. -- P.677–686.
  \bibitem{ShaPhD} Shareshian J. Combinatorial properties of subgroup lattices of finite groups. -- Ph.D.Thesis. Rutgers University, 1996.
  \bibitem{Quillen} Quillen D. Homotopy properties of the poset of nontrivial $p$-subgroups of a group. // Adv. in Math. -- 1978. -- Vol. 28, №2. -- P.101-128.
  \bibitem{Ramras} Ramras D.A. Connectivity of the coset poset and the subgroup poset of a group // J. Group Theory. -- 2005. -- Vol. 8. -- P.719–746.
  \bibitem{Suzuki} Suzuki M. On a class of doubly transitive groups // Annals of Math. -- 1962. -- Vol. 75. -- P.105-145.
  \bibitem{FomenkoFuks} Fomenko A.T., Fucks D.B. A course in homotopy topology. -- Moscow: Nauka, 1989.
\end{thebibliography}
\end{document}